\documentclass{mmn_sample}

\theoremstyle{plain}
\newtheorem{theorem}{Theorem}
\newtheorem{corollary}{Corollary}
\newtheorem{lemma}{Lemma}
\newtheorem{proposition}{Proposition}

\theoremstyle{definition}

\theoremstyle{remark}

\numberwithin{equation}{section}

\allowdisplaybreaks[4]

\begin{document}
\title[Non-autonomous Non-instantaneous Impulsive System
]{On the existence and uniqueness of solutions for non-autonomous semi-linear systems with non-instantaneous impulses, delay, and non-local conditions
} 

\author[S. Lalvay]{Sebasti\'an Lalvay}
\address{Universidad Yachay Tech\\ School of Mathematical and Computational Sciences\\ Hacienda San Jos\'e S/N\\ 100115
  Urcuqu\'i\\ Imbabura, Ecuador}

\email{sebastian.lalvay@yachaytech.edu.ec}

\thanks{The first author was supported in part by the XXX Fund, Grant No.~YYYYY.}

\author[A. Padilla-Segarra]{Adri\'an Padilla-Segarra}
\address{Universidad Yachay Tech\\ School of Mathematical and Computational Sciences\\ Hacienda San Jos\'e S/N\\ 100115
	Urcuqu\'i\\ Imbabura, Ecuador}

\email{adrian.padilla@yachaytech.edu.ec}

\author[W. Zouhair]{Walid Zouhair}

\address{Laboratory of Mathematics and Population Dynamics\\ Faculty of Sciences of Semlalia, Cadi Ayyad University\\ Marrakesh, BP 2390, 40000\\ Morocco}


\email{walid.zouhair.fssm@gmail.com}

\begin{abstract}
A non-autonomous evolution semi-linear differential system under non-instantaneous impulses, delays, and perturbed by non-local conditions is studied. Its piece-wise continuous solutions belong to a finite dimensional Banach space. The existence and uniqueness of solutions on the interval $[-r,\tau]$ are obtained by applying Karakostas' fixed-point theorem. Further results concerning solution prolongation are developed. An example is presented, and several remarks on the infinite-dimensional case are included.
\end{abstract}




\maketitle

\section{Introduction}
Impulsive systems are of vital importance on most scientific fields. They can  be found in applications ranging from biology and population dynamics to economics and engineering. Usually, many situations are modeled by differential equations. Controls, delays, impulses, and non-linear perturbations are added to capture either feedback or activity characterization. 

The interest of this article is the  non-autonomous non-instantaneous impulsive semi-linear system involving state-delay and non-local conditions, which is motivated by applications, such as species population, nanoscale electronic circuits consisting of single-electron tunneling junctions, and mechanical systems with impacts \cite{Lakshmikantham1989,Samoilenko1995,Yang2001}. In particular, impulses represent sudden deviations of the states at specific times, by either instantaneous jumps or continuous intervals.

Our mathematical motivation is to extend the existence and uniqueness of solutions on a finite-dimensional Banach space  \cite{Abada2010,Balachandran1996,Muslim2018} for the aforementioned semi-linear system. It is worth to highlight that some existence and controllability results on the impulsive autonomous case have been done by \cite{Nieto2010,Pandey2014,Wang2015}, and \cite{Leiva2017}. The latter authors are the pioneer in implementing Karakostas' fixed-point theorem to prove the existence of solutions on semi-linear equations with instantaneous impulses.

Furthermore, techniques from \cite{Bashirov2015} and Rothe's fixed-point theorem have been used in \cite{Guevara2018W,Guevara2018H,Malik2019} for studying the approximate and exact controllability of this family of systems with delays, instantaneous impulses, and memory considerations. Important results on autonomous impulsive systems involving delay were developed by \cite{Driver:1977,Hernandez2006,Li2006}. In addition, problems with non-local conditions including impulses can be found in \cite{Leiva2018}.

Finally, \cite{Hernandez2013} introduced the class of non-instantaneous impulsive systems, and \cite{Pierri2013} showed the existence of solutions for these systems. Later, \cite{A.Anguraj2016,Agarwal2017b,LeivaZouhair2020} showcased relevant studies for models on  non-instantaneous impulsive differential equations. However, it is not of our knowledge that there are results on the existence of solutions for semi-linear non-autonomous systems including all conditions simultaneously. This is the center of our research work.


This article is structured as follows. Section \ref{System Description} describes the  analyzed system and notation. Section \ref{Preliminary Theory and Hypotheses} deals with preliminary concepts, definitions, and hypotheses used throughout this work. Section \ref{Existence and Uniqueness of Solutions} is devoted to the existence and uniqueness of solutions for the system in the light of  Karakostas' fixed-point theorem, which is an extension of the fixed-point theorem due to M. A. Krasnosel’ski{\v\i} developed in \cite{Karakostas2003}. Finally, sections \ref{Example} and \ref{Final Remarks} illustrate these results with an example of the considered system and present conclusions and guidelines for open problems.
 


\section{System Description}\label{System Description}
Let $N\in\mathbb{N}$, and denote $I_N$ as the set $\{1,2,\ldots,N\}$. In this article, the existence and uniqueness of solutions for the following semi-linear non-autonomous system are proved:
\begin{equation}\label{eqExist}
	\begin{cases}
		z^{\prime}(t)  =  \boldsymbol{A}(t)z(t) +f(t,z_{t}), &t \in \bigcup\limits_{i=0}^{N}\left(s_{i}, t_{i+1}\right], \\
		z(t)= G_{i}(t,z(t)), &  t\in(t_{i},s_{i}],\ i\in I_N,\\
		z(t)=\phi(t)- g(z_{\theta_1},z_{\theta_2},\dots, z_{\theta_q})(t), & t \in [-r,0],
	\end{cases}
\end{equation}
where $s_{i}, t_{i}, \theta_{j}, r \in  (0,\tau)$,    with $t_{i}\leq s_{i}<t_{i+1}$, $\theta_{j}< \theta_{j+1}$, for $i\in I_N$ and $j\in I_{q}$, $s_{0}=t_{0}=0$, and  $t_{N+1}=\theta_{q+1}=\tau$, all fixed real numbers. The system solutions are denoted by $z:\mathcal{J}=[-r,\tau]\longrightarrow \mathbb{R}^n$
and the non-instantaneous impulses are represented by $G_i:(t_i,s_i]\times\mathbb{R}^{n}\longrightarrow \mathbb{R}^n,\ i\in I_N.$ $\boldsymbol{A}$ is a continuous matrix such that $\boldsymbol{A}(t)\in\mathbb{R}^{n\times n}$, $t \in \mathbb{R}$. $z_{t}$ stands for the translated function of $z$ defined by $z_{t}(s) = z(t+s)$, with $s\in [-r,0]$. The function $f :\overline{\mathbb{R}_+}\times PC_{r}([-r,0]; \mathbb{R}^{n}) \longrightarrow \mathbb{R}^n$ represents the non-linear perturbation of the differential equation in the system, where $\overline{\mathbb{R}_+}=[0,+\infty)$, and $g: PC_{r}^q([-r,0]; (\mathbb{R}^{n})^{q}) \longrightarrow PC_{r}([-r,0]; \mathbb{R}^{n})$ indicates the behavior in the non-local conditions. The function
\begin{equation}\label{phi}
    \phi:[-r,0]\longrightarrow\mathbb{R}^n
\end{equation}
represents the historical pass on the time interval $[-r,0]$.

In order to properly set system \eqref{eqExist}, the following Banach spaces are considered. Denote $C(\mathcal{U};\mathbb{R}^n)$ as the space of continuous functions on a set $\mathcal{U}\subset \mathbb{R}$. $PC_{r}=PC_{r}([-r,0];\mathbb{R}^n)$ is the space of continuous functions of the form \eqref{phi} except on a finite number of points $r_i,\ i\in I_l$, with \ $l \leq N$, where the side limits $\phi(r_i^+),\ \phi(r_i^-)$ exist, and $\phi(r_i)=\phi(r_i^-)$, for all $i\in I_N$, endowed with the supremum norm.

The natural Banach space for the solutions of system \eqref{eqExist} is defined as:
\begin{equation*}
	\begin{split}
		PC_{r\tau}=PC_{r\tau}(\mathcal{J};\mathbb{R}^n)= \Big\{z: & \mathcal{J} \longrightarrow  \mathbb{R}^{n}\ {\Big |} \ z\big |_{[-r,0]}\in PC_{r},\ z\big |_{[0,\tau]} \in C(\mathcal{J}^{\prime}; \mathbb{R}^{n}),\\
		&\text{ there exist }\ z(t_{k}^{+}),\ z(t_{k}^{-}),\ \text{ and }\ z(t_{k})=z(t_{k}^{-}),\ k\in I_N \Big\},
	\end{split}
\end{equation*}
where $\mathcal{J}^{\prime} = [0,\tau] \backslash \{t_1, t_2, \dots, t_{_{N}} \}$, and endowed with the norm
\begin{equation*}
	\|z\|=\|z\|_{0} = \sup_{t \in  \mathcal{J}} \|z(t)\|_{\mathbb{R}^n},\quad z\in PC_{r\tau}.
\end{equation*}

The cartesian product space given by $	\left(\mathbb{R}^n\right)^q=\mathbb{R}^n\times \mathbb{R}^n\times \dots \times \mathbb{R}^n=\displaystyle\prod_{i=1}^q \mathbb{R}^n$ is equipped with the norm $\displaystyle\|z\|_{\left(\mathbb{R}^n\right)^q}=\sum_{i=1}^q \|z_{i}\|_{\mathbb{R}^n}$, for $ z\in \left(\mathbb{R}^n\right)^q.$

The space $PC_{r}^q=PC_{r}^q([-r,0]; (\mathbb{R}^{n})^{q})$ is defined analogously and endowed with the norm
\begin{equation*}
	\|z\|_{PC_{r}^q}=\sup_{t\in [-r,0]}\|z(t)\|_{\left(\mathbb{R}^n\right)^q},\quad z\in PC_{r}^q.
\end{equation*}

\section{Preliminary Theory and Hypotheses}\label{Preliminary Theory and Hypotheses}

In this section, the evolution operator based on the corresponding linear system is defined. This work can be extended to infinite-dimensional Banach spaces. Thus, the properties of the evolution operator are included, aiming to establish their similarities with the latter case,  where uniform continuity is lost unless the evolution operator is assumed to be compact. Finally, the system solutions are characterized, and hypotheses for applying Karakostas' fixed-point theorem are presented.

Let $\boldsymbol{U}$  be the evolution operator corresponding to system \eqref{eqExist}
\begin{equation}\label{evolutionOp}
	\boldsymbol{U}(t,s) = \Phi(t)\Phi^{-1}(s),\quad\text{ for all } t, s\in \mathbb{R},
\end{equation}
where $\Phi$ is the fundamental matrix solution of the associated linear system
\begin{equation}\label{uncontrolled}
	z^{\prime}(t) =  \boldsymbol{A}(t)z(t).
\end{equation}


		
Therefore, there exist constants $\widehat{M},\ \omega >0$ and $M \geq 1$ such that:
		\begin{equation*}
			\|\boldsymbol{U}(t, s)\| \leq \widehat{M} e^{\omega(t-s)}\leq M, \quad 0 \leq s \leq t \leq \tau,
		\end{equation*}

The following proposition exhibits a characterization of the system \eqref{eqExist} solutions and is based on the works done in \cite{Leiva2018} and \cite{Pierri2013}.

\begin{proposition}
	
	\label{characterization}
	The semi-linear system \eqref{eqExist} has a solution $z\in PC_{r\tau}(\mathcal{J};\mathbb{R}^n)$ if, and only if,
	\begin{equation}
		\label{solution}
		z(t)=\begin{cases}
			\boldsymbol{U}(t,0)[\phi(0)-g(z_{\theta_1},z_{\theta_2},\dots, z_{\theta_q})(0)]\\
			+ \displaystyle \int_{0}^{t}\boldsymbol{U}(t,s)f(s,z_{s})ds,  & t \in (0,t_{1}],\\
			\displaystyle \boldsymbol{U}(t,s_{i})G_{i}(s_{i},z(s_{i})) +\displaystyle \int_{s_{i}}^{t} \boldsymbol{U}(t,s) f(s,z_{s})ds,
			& t \in\left(s_{i}, t_{i+1}\right],\ i\in I_N,\\
			G_{i}(t,z(t)), & t \in (t_{i},s_{i}],\ i\in I_N,\\
			\phi(t)- g(z_{\theta_1},z_{\theta_2},\dots, z_{\theta_q})(t), \quad  & t\in [-r,0].
		\end{cases}
	\end{equation}
\end{proposition}

Observe that on some interval $[-r,p_1)$, if a solution $z$ of the form \eqref{solution} is defined, 
and there is not $p_2>p_1$ such that a solution can be defined on $[-r,p_2)$, then, $[-r,p_1)$ is a \textit{maximal interval} of existence. 

In this work, the following hypotheses are assumed:
\begin{enumerate}
	\item[\textbf{H1}] The next conditions hold:
	\begin{enumerate}
		\item[(i)] The function $g$ fulfills that $g(0)=0$, and there exists $N_q > 0$ such that, for all $y,\ z \in PC_r^q$ and $t\in [-r,0]$,
		\begin{equation*}
			\|g(y)(t)-g(z)(t)\|_{\mathbb{R}^n} \leq N_{q} \|y(t)-z(t)\|_{(\mathbb{R}^n)^q}.
		\end{equation*}
		
		\item[(ii)] There exists a constant $L >0$ such that, for all $i\in I_N$, the functions $G_i$ satisfy $G_i(\cdot,0) = 0,$
		and, if $\varphi_1,\varphi_2\in PC_{r\tau}$, for $t\in (t_i,s_i]$, then,
		\begin{equation*}
			\|G_i(t,\varphi_1(t))-G_i(t,\varphi_2(t))\|_{\mathbb{R}^n}\leq L\|\varphi_1-\varphi_2\|,\quad\text{where}\quad L + N_{q} q  < \frac{1}{2}.
		\end{equation*}
	\end{enumerate}
	
	\item[\textbf{H2}] The function $f$ satisfies the following conditions:
	\begin{align*}
		\|f(t,\varphi_1)-f(t,\varphi_2)\|_{\mathbb{R}^n}  &\leq K(\|\varphi_1\|, \|\varphi_2\| )\|\varphi_1-\varphi_2\|,\\
		\|f(t,\varphi)\|_{\mathbb{R}^n}   &\leq  \Psi(\|\varphi\|),
	\end{align*}
	where $K: \overline{\mathbb{R}_{+}} \times \overline{\mathbb{R}_{+}} \longrightarrow \overline{\mathbb{R}_{+}}$ and $\Psi: \overline{\mathbb{R}_{+}} \longrightarrow \overline{\mathbb{R}_{+}}$ are continuous and non-decreasing functions in their arguments, and $\varphi, \varphi_1, \varphi_2 \in  PC_{r}([-r,0];\mathbb{R}^n)$.
	
	\item[\textbf{H3}] The following relations hold for $\tau$ and $\rho >0$:
	\begin{enumerate}
		\item[(i)] $ MN_{q} q \left(\|\tilde{\phi}\|+\rho\right)+ M \tau \Psi\left(\|\tilde{\phi}\| +\rho \right) \leq \rho, $
		\item[(ii)] $ML\left(\|\tilde{\phi}\|+\rho\right) +\| \alpha \|_{\mathbb{R}^n} + M \tau\Psi\left(\|\tilde{\phi}\|+\rho\right)\leq \rho, $
		\item[(iii)] $\displaystyle L\left(\|\tilde{\phi}\|+\rho\right) + \|\beta\|_{\mathbb{R}^n}  \leq \rho, $ 
	\end{enumerate}
	where $\alpha, \beta \in \mathbb{R}^n$ are arbitrarily fixed, and the function $\tilde{\phi}$ is defined as:
	\begin{equation}\label{phi_tilde}
		\tilde{\phi}(t)= \begin{cases}
			\boldsymbol{U}(t,0)\phi(0), & t \in (0, t_1], \\
			\alpha, & t \in \bigcup\limits_{i=1}^N (s_i,t_{i+1}],\\
			\beta,   & t \in\bigcup\limits_{i=1}^N (t_i,s_i],\\
			\phi(t), &t \in [-r,0].
		\end{cases}
	\end{equation}
	
	\item[\textbf{H4}] The following relations hold for $\tau$ and $\rho >0$:
	\begin{enumerate}
		\item[(i)] $M N_{q} q+ M \tau K\left(\|\tilde{\phi}\|+\rho,\|\tilde{\phi}\| +\rho \right)<1,$
		\item[(ii)]$\displaystyle ML + M \tau K\left(\|\tilde{\phi}\|+\rho, \|\tilde{\phi}\|+\rho \right)  < 1.$
	\end{enumerate}
\end{enumerate}

Theorem \ref{Karakostas} is the statement of Karakostas' fixed-point theorem (see \cite{Karakostas2003, Leiva2017}).
\begin{theorem}[Karakostas]\label{Karakostas}
	Let $\mathcal{P}$ and $\mathcal{Q}$ be Banach spaces, $D \subset \mathcal{P}$ a closed and convex subset, and $J:D \longrightarrow \mathcal{Q}$ a continuous compact operator. Let $F: D \times \overline{J(D)}\longrightarrow D$ be a continuous operator such that the family given by \newline $\left\{F(\cdot, y): y\in \overline{J(D)}\right\}$ is equicontractive. Then, $F(z, J(z))=z$ admits a solution in $D$.
\end{theorem}

\section{Existence and Uniqueness of Solutions}\label{Existence and Uniqueness of Solutions}

In this section, the proofs of the existence and uniqueness of the solution for system \eqref{eqExist} are presented. To apply Karakostas' fixed-point theorem, the operators $J$ and $F$ are defined. Then, a fixed point on a subset of $PC_{r\tau}$ for equation \eqref{fp} is proved. Therefore, the problem of finding a solution of the form \eqref{solution} becomes a fixed-point problem.

Consider the following continuous operators
\begin{align*}
	J:PC_{r\tau}(\mathcal{J};\mathbb{R}^n) &\longrightarrow PC_{r\tau}(\mathcal{J};\mathbb{R}^n),\\
	F: PC_{r\tau}(\mathcal{J};\mathbb{R}^n) \times PC_{r\tau}&(\mathcal{J};\mathbb{R}^n) \longrightarrow PC_{r\tau}(\mathcal{J};\mathbb{R}^n),
\end{align*}
and a fixed $\eta \in\mathbb{R}^n$. For $y,\ z\in PC_{r\tau}$,
\begin{equation}\label{OperatorJ}
	J(y)(t)=\begin{cases}
		\displaystyle \boldsymbol{U}(t,0)\left[\phi(0)-g(y_{\theta_1},y_{\theta_2},\dots, y_{\theta_q})(0)\right]\\
		+ \displaystyle\int_{0}^{t}\boldsymbol{U}(t,s)f(s,y_{s})ds, \quad &  t \in (0,t_1],\\
		\boldsymbol{U}(t,s_i)G_i(s_i,y(s_i)) \\
		\displaystyle + \int_{s_i}^{t}\boldsymbol{U}(t,s)f(s,y_{s})ds,\quad &  t \in (s_i,t_{i+1}],\ i\in I_N ,\\
		\eta,  \quad &t \in \bigcup\limits_{i=1}^N (t_i,s_i],\\
		\phi(t),  \quad &t \in [-r,0],
	\end{cases}
\end{equation}

\begin{equation}\label{OperatorF}
	F(z,y)(t)= \begin{cases}
		y(t) ,  \quad &t\in \bigcup\limits_{i=0}^N (s_i,t_{i+1}],\\
		G_i (t,z(t)),  \quad &t\in (t_i,s_{i}],\ i\in I_N ,\\
		\phi(t)-g(z_{\theta_1},z_{\theta_2},\dots, z_{\theta_q})(t),  \quad &t \in [-r,0].
	\end{cases}
\end{equation}

From the $J$ and $F$ definition, the following fixed-point equation is equivalent to solve system \eqref{eqExist}:
\begin{equation} \label{fp}
	F(z,J(z)) = z,\quad z\in PC_{r\tau}.
\end{equation}

First, it is observed that $J$ is compact, and the set $\left\{F(\cdot, y)\ :\ y \in \overline{J(D_{\rho})}\right\}$ is equicontractive, where
\begin{equation}\label{Dset}
	D_{\rho}=D_{\rho}(\tau, \phi):= \left\{\varphi \in PC_{r\tau}(\mathcal{J};\mathbb{R}^n)\ :\ \|\varphi - \tilde{\phi}\| \leq \rho\right\},\quad \text{ for } \rho >0.
\end{equation}

This set is closed and convex, and $\tilde{\phi}$ is given by \eqref{phi_tilde}. So, the hypotheses of Theorem \ref{Karakostas} are satisfied. Lemma \ref{lemma} highlights the relevance of hypotheses \textbf{H1}, \textbf{H2}, and how they fit into the main results. Theorems \ref{existence}, \ref{uniqueness} and \ref{prolongation} follow on this foundation. 

\begin{lemma} \label{lemma}
	Let hypotheses \textbf{H1} and \textbf{H2} be satisfied. Then, the operators $J$ and $F$ satisfy the following assertions:
	\begin{enumerate}
		\item[(i)] $J$ is continuous.
		\item[(ii)] $J$ maps bounded sets onto bounded sets.
		\item[(iii)] $J$ maps bounded sets onto equicontinuous sets.
		\item[(iv)] $J$ is a compact operator.
		\item[(v)] The set $\left\{F(\cdot,y): y\in\overline{J(D_{\rho})}\right\}$ is comprised of equicontractive operators, with $D_{\rho}$ as in \eqref{Dset}.
	\end{enumerate}
\end{lemma}

\begin{proof}
	\begin{enumerate}
		\item[(i)] \textit{$J$ is continuous.}\vspace{.3cm}
		
		Taking $y,\ z \in PC_{r\tau}$, trivially, for $t\in[-r,0]$,
		\begin{equation*}
			\|J (z)(t)-J(y)(t)\|_{\mathbb{R}^n} =\|\phi(t)-\phi(t)\|_{\mathbb{R}^n}=0.
		\end{equation*}
		
		Thus,
		\begin{equation}\label{r0}
			\sup_{t\in [-r,0]} \|J (z)(t)-J(y)(t)\|_{\mathbb{R}^n}= 0.
		\end{equation}
		
		By \textbf{H1}, \textbf{H2}, and $t \in (0,t_1]$, the following estimate holds:
		\begin{equation*}
			\begin{split}
				\|J(z)(t)-J(y)(t)\|_{\mathbb{R}^n}
				& \leq M\|g(y_{\theta_1},\dots, y_{\theta_q})(0)-g(z_{\theta_1},\dots, z_{\theta_q})(0)\|_{\mathbb{R}^n}\\
				& + M \int_{0}^{t}\|(f(s,z_s)-f(s,y_s))\|_{\mathbb{R}^n}ds\\
				& \leq MN_{q}\sum_{i=1}^q\|y-z\|\\ 
				& + M \int_{0}^{t}K\left(\|z_s\|,\|y_s\|\right)\|z_s-y_s\|ds\\
				& \leq MN_{q} q \|z-y\|+M t_1 K(\|z\|,\|y\|)\|z-y\|.
			\end{split}
		\end{equation*}
		
		Taking the sup,
		\begin{equation}\label{r1}
			\sup_{t\in (0,t_1]}\|J(z)(t)-J(y)(t)\|  \leq  M\left[ N_{q} q+ t_1K(\|z\|,\|y\|)\right]\|z-y\|.
		\end{equation}
		
		Now, for $i\in I_N$ and $t\in (s_i,t_{i+1}]$,
		\begin{equation*}
			\begin{split}
				\|J(z)(t)-J(y)(t) \|_{\mathbb{R}^n}
				& \leq M\|G_i(s_i,z(s_i))- G_i(s_i,y(s_i) )\|_{\mathbb{R}^n}\\
				& + M \int_{s_i}^{t}\|(f(s,z_s)-f(s,y_s))\|_{\mathbb{R}^n}ds\\
				& \leq M  L \|z(s_i)-y(s_i)\|_{\mathbb{R}^n}\\
				& + M \int_{s_i}^{t}K(\|z_s\|,\|y_s\|)\|z_s-y_s\|ds\\
				& \leq M L \|z-y\|+M \tau K(\|z\|,\|y\|)\|z-y\|.
			\end{split}
		\end{equation*}
		
		Thus,
		\begin{equation*}
			\sup_{t\in (s_i,t_{i+1}]}\|J(z)(t)-J(y)(t)\|  \leq M\left[L + \tau K(\|z\|,\|y\|) \right]\|z-y\|.
		\end{equation*}
		
		Together with \eqref{r0}, \eqref{r1}, and since $J$ is constant on $\bigcup\limits_{i=1}^{N} (t_i,s_i]$, it yields that there exists $N_{y,z}>0$ such that:
		\begin{equation*}
			\|J(z)-J(y)\|\leq N_{y,z}\|z-y\|.
		\end{equation*}
		
		Hence, $J$ is continuous. And, in fact, it is Lipschitz continuous.\vspace{.3cm}
		
		\item[(ii)] \textit{$J$ maps bounded sets onto bounded sets.}\vspace{.3cm}
		
		Without loss of generality, set $R >0$ arbitrarily and prove that there exists $d >0$ such that, for every $y \in B_{R} =\overline{B_{R}(0)}=\left\{ z \in PC_{r\tau}\ :\ \|z\| \leq R\right\}$, it follows that $\|J(y) \| \leq d$.
		
		For $t\in[-r,0]$, it gives that:
		\begin{equation*}
			\|J(y)(t) \|_{\mathbb{R}^n} = \|\phi(t)\|_{\mathbb{R}^n} \leq \|\phi\|=:d_0.
		\end{equation*}
		
		Let $y \in B_{R}$ and $t \in (0,t_1]$. \textbf{H2} yields
		\begin{equation*}
			\begin{split}
				\|J(y)(t) \|_{\mathbb{R}^n} & \leq \left\|\boldsymbol{U}(t,0)\left[\phi(0)-g(y_{\theta_1},\dots, y_{\theta_q})(0)\right]\right\|_{\mathbb{R}^n}\\
				&+ \int_{0}^{t}\|\boldsymbol{U}(t,s)f(s,y_s)\|_{\mathbb{R}^n}ds\\
				& \leq M \left(\|\phi(0)\|_{\mathbb{R}^n}+ N_{q} q\|y\|\right) + M t_1\Psi(\|y\|)\\
				& \leq M\left(\|\phi(0)\|_{\mathbb{R}^n}+ N_{q} qR\right) + M t_1\Psi(R)=:d_1.
			\end{split}
		\end{equation*}
		
		Similarly,  for each $i\in I_N$, if $t\in (s_i,t_{i+1}]$, then,
		\begin{equation*}
			\begin{split}
				\|J(y)(t) \|_{\mathbb{R}^n} & \leq \|\boldsymbol{U}(t,s_i)G_i(s_i,y(s_i))\|_{\mathbb{R}^n} + \int_{s_i}^{t}\|\boldsymbol{U}(t,s)f(s,y_s)\|_{\mathbb{R}^n}ds\\
				& \leq ML\|y\|+ M (t_{i+1}-s_i)\Psi(\|y||)\\
				& \leq M LR + M \tau \Psi(R)=:d_2.
			\end{split}
		\end{equation*}
		
		Finally, whenever $t\in (t_i,s_i]$, for $i\in I_N$, it follows that:
		\begin{equation*}
			\|J(t)\|_{\mathbb{R}^n}=\|\eta\|=:d_3.
		\end{equation*}
		
		Taking $d=\displaystyle\max_{0\leq i\leq3}\{d_i\}$, boundedness is proved.\vspace{.3cm}
		
		\item[(iii)] \textit{$J$ maps bounded sets onto equicontinuous sets.}\vspace{.3cm}
		
		Let $B_{R}$ as in (ii), and $y\in B_R$, arbitrary. For some $ 0<\nu_1 <\nu_2\leq t_1$, the following estimate holds:
		\begin{equation}\label{eq3.a}
			\begin{split}
				\|J(y)(\nu_2)&-J(y)(\nu_1) \|_{\mathbb{R}^n}\\
				&\leq \left\|\left[\boldsymbol{U}(\nu_2,0)-\boldsymbol{U}(\nu_1,0)\right]\left[\phi (0)-g(y_{\theta_1},\ldots,y_{\theta_q})(0)\right]\right\|_{\mathbb{R}^n} \\
				&+\left\|\int_{0}^{\nu_2}\boldsymbol{U}(\nu_2,s)f(s,y_s)ds-\int_{0}^{\nu_1}\boldsymbol{U}(\nu_1,s)f(s,y_s)ds\right\|_{\mathbb{R}^n} \\
				& \leq \left\|\boldsymbol{U}(\nu_2,0)-\boldsymbol{U}(\nu_1,0)\right\|\left(\| \phi(0)\|_{\mathbb{R}^n}+N_{q}\displaystyle\sum_{i=1}^q \|y_{\theta_i}(0)\| _{\mathbb{R}^n}\right) \\
				&  +  \int_{0}^{\nu_1}\left\|\left(\boldsymbol{U}(\nu_2,s)-\boldsymbol{U}(\nu_1,s)\right)f(s,y_s)\right\|_{\mathbb{R}^n}ds \\
				& +  \int_{\nu_1}^{\nu_2}\|\boldsymbol{U}(\nu_2,s)f(s,y_s)\|_{\mathbb{R}^n}ds  \\
				& \leq  \left\|\boldsymbol{U}(\nu_2,0)-\boldsymbol{U}(\nu_1,0)\right\| \left(\| \phi(0) \|_{\mathbb{R}^n}+N_{q}qR\right) \\
				& + \Psi(R)\int_{0}^{\nu_1}\|\boldsymbol{U}(\nu_2,s)-\boldsymbol{U}(\nu_1,s)\|ds +  M\Psi (R)(\nu_2-\nu_1).
			\end{split}
		\end{equation}
		
		Similarly, for each $i\in I_N$ and every $\nu_1,\nu_2$, with $s_i <\nu_1 <\nu_2 \leq t_{i+1}$, it follows that:
		\begin{equation}
			\begin{split}
				\|J(y)(\nu_2)&-J(y)(\nu_1) \|_{\mathbb{R}^n}\\
				& \leq \|\boldsymbol{U}(\nu_2,s_i)-\boldsymbol{U}(\nu_1,s_i)\| \| G_i(s_i, y(s_i)) \|_{\mathbb{R}^n} \\
				& + \left\|\int_{s_i}^{\nu_2} \boldsymbol{U}(\nu_2,s)f(s,y_s)ds-\int_{s_i}^{\nu_1}\boldsymbol{U}(\nu_1,s)f(s,y_s)ds\right\|_{\mathbb{R}^n}  \\
				& \leq \|\boldsymbol{U}(\nu_2,s_i)-\boldsymbol{U}(\nu_1,s_i)\|  \, L\| y\|  \\
				&  + \int_{s_i}^{\nu_1}\|\boldsymbol{U}(\nu_2,s)-\boldsymbol{U}(\nu_1,s)\| \Psi(\| y_s \|) ds + M  \int_{\nu_1}^{\nu_2} \Psi(\| y_s \|) ds \\
				& \leq  \|\boldsymbol{U}(\nu_2,s_i)-\boldsymbol{U}(\nu_1,s_i)\| L R  \\
				&+ \Psi(R)\int_{s_i}^{\nu_1}\|\boldsymbol{U}(\nu_2,s)-\boldsymbol{U}(\nu_1,s)\|ds + M\Psi(R)(\nu_2-\nu_1). \label{eq3.b}
			\end{split}
		\end{equation}
		
		By \eqref{eq3.a} and \eqref{eq3.b}, the continuity and boundedness of $\boldsymbol{U}(t,s)$ yield that, as $\nu_2$ approaches to $\nu_1$, $\|J(y)(\nu_2)-J(y)(\nu_1) \|_{\mathbb{R}^n}$ goes to zero, independently of $y$. Therefore, $J(B_{R})$ is equicontinuous on the set $\bigcup\limits_{i=0}^N (s_i,t_{i+1}]$. In the same fashion, equicontinuity on $[-r,0]$ and $\bigcup\limits_{i=1}^N (t_i,s_i]$ is obtained. And, the family of functions $J(B_{R})$ is equicontinuous on the interval $\mathcal{J}\backslash \{t_1,\ldots,t_N\}$.\vspace{.3cm}
		
		\item[(iv)] \textit{$J$ is a compact operator.}\vspace{.3cm}
		
		Let $B\subset PC_{r\tau}$ be a bounded subset, and $\left\{\omega_{n}\right\}_{n\in\mathbb{N}}$, a sequence on $J(B)$. Then, (ii) and (iii) imply that it is uniformly bounded and equicontinuous on $[-r, t_1]$. Note that $\left\{\omega_n|_{[-r,0]} \right\}_{n\in\mathbb{N}} = \{\phi \}$. Arzelà-Ascoli theorem on \newline $\left\{\omega_n|_{[0,t_1]} \right\}_{n\in\mathbb{N}}\subset C\left([0,t_1];\mathbb{R}^n\right)$ implies there is a uniformly convergent subsequence $\left\{\omega^{1}_{n}\right\}_{n\in\mathbb{N}}$ on $[-r, t_1]$.
		
		Consider the sequence $\left\{\omega^{1}_{n}\right\}_{n\in\mathbb{N}}$ on the interval $[s_1, t_2]$. It is uniformly bounded and equicontinuous, and as before, it has a convergent subsequence $\left\{\omega^{2}_{n}\right\}_{n\in\mathbb{N}}$ on $[s_1, t_2]$. Therefore, a uniformly convergent subsequence $\left\{\omega^{2}_{n}\right\}_{n\in\mathbb{N}}$ of $\left\{ \omega_n \right\}_{n\in\mathbb{N}}$ on the interval $[-r,t_2]$ is obtained, since each $\omega^2_n$ has the same definition on $[t_1,s_1]$.
		
		Continuing this process on the intervals $[t_2,s_2],\,[s_2, t_3],\, [t_3, s_3], \ldots, [s_N, \tau]$, it is concluded that there is a subsequence $\left\{\omega^{N+1}_{n}\right\}_{n\in\mathbb{N}}$ of $\left\{ \omega_n\right\}_{n\in\mathbb{N}}$, uniformly convergent on $[-r, \tau]$. Thus, the set $\overline{J(B)}$ is compact, and by the characterization of sequentially compact spaces, $J$ is compact.\vspace{.3cm}
		
		\item[(v)] \textit{The set $\left\{F(\cdot, y)\ :\ y \in  \overline{J(D_{\rho})}\right\}$ is comprised of equicontractive operators.}\vspace{.3cm}
		
		Let $\rho>0$, $y\in \overline{J(D_{\rho})}$, $\ x,\ z \in PC_{r\tau}$, and $t \in [-r, 0]$. Thus, \textbf{H1} yields
		\begin{equation}\label{eq5.a}
			\begin{split}
				\|F(z, y)(t)-F(x, y)&(t) \|_{\mathbb{R}^n}\\
				& \leq \left\|g(x_{\theta_1},\dots, x_{\theta_q})(t)-g(z_{\theta_1},\dots, z_{\theta_q})(t)\right\|_{\mathbb{R}^n} \\
				& \leq N_{q} q\|z-x\|.
			\end{split}
		\end{equation}
		
		For each $i\in I_N$ and $t \in (t_i, s_i]$, it follows that:
		\begin{equation} \label{eq5.b}
			\begin{split}
				\|F(z, y)(t)-F(x, y)(t) \|_{\mathbb{R}^n}
				& \leq \|G_i(t,z(t))-G_i(t,x(t))\|_{\mathbb{R}^n} \\
				& \leq L \|z-x \|.
			\end{split}    
		\end{equation}
		
		Moreover, on the intervals $(s_i,t_{i+1}],\ i\in\{ 0\}\cup I_N$, it follows that:
		\begin{equation}\label{eq5.c}
			\|F(z, y)(t)-F(x, y)(t)\|_{\mathbb{R}^n} =\|y(t)-y(t)\|_{\mathbb{R}^n}=0.
		\end{equation}
		
		Combining \eqref{eq5.a}-\eqref{eq5.c}, the next estimate holds:
		\begin{equation*}
			\|F(z, y)-F(x, y) \| \leq\frac{1}{2}\|z-x \|.
		\end{equation*}
		
		Hence, $F$ is a contraction on the first variable, independently of $y \in \overline{J(D_{\rho})}$.
	\end{enumerate}

\end{proof}

\begin{theorem}\label{existence}
	Assume \textbf{H1} - \textbf{H3}. Then, problem \eqref{eqExist} has, at least, one solution on the interval $\mathcal{J}=[-r,\tau]$.
\end{theorem}

\begin{proof}
	For $\rho>0$, let $D_{\rho}$ as in \eqref{Dset}, and define the operators $\widetilde{J}$ and $\widetilde{F}$ as:
	\begin{equation*}
		\widetilde{J}=J\big|_{D_{\rho}}:D_{\rho}\longrightarrow PC_{r\tau}(\mathcal{J};\mathbb{R}^n)\quad\text{and}\quad \widetilde{F}=F\big|_{D_{\rho}\times\overline{\widetilde{J}(D_{\rho})}}:D_{\rho}\times\overline{\widetilde{J}(D_{\rho})}\longrightarrow D_{\rho}.
	\end{equation*}
	
	Because of Lemma \ref{lemma}, $\widetilde{J}$ is continuous and compact, and the family\\ $\left\{F(\cdot, y)\ :\ y \in  \overline{J(D_{\rho})}\right\}$ is equicontractive. Continuity of $\widetilde{F}$ follows analogously. The goal is to prove that, indeed, $\widetilde{F}\left(D_{\rho}, \overline{\widetilde{J}(D_{\rho})}\right) \subset D_{\rho}.$ Thus, Theorem \ref{Karakostas} assumptions will be satisfied, and an equivalent solution will be obtained.
	
	Take an arbitrary $z \in D_{\rho}$, for $t\in[-r,0]$, it yields
	\begin{equation} \label{eqT.a}
		\begin{split}
			\left\|\widetilde{F}\left(z,\widetilde{J}(z)\right)(t)-\tilde{\phi}(t)\right\|_{\mathbb{R}^n}&= \|g(z_{\theta_1},\ldots,z_{\theta_q})(t)\|_{\mathbb{R}^n} \\
			&\leq N_q\sum_{j=1}^q\|z_{\theta_j}(t)\|_{\mathbb{R}^n} \\
			&\leq MN_qq\|z\| \\
			&\leq MN_qq(\|\tilde{\phi}\|+\rho) \leq \rho. 
		\end{split}
	\end{equation}
	
	Similarly, $t \in (0,t_1]$ imply
	\begin{equation} \label{eqT.b}
		\begin{split}
			\left\|\widetilde{F}\left(z,\widetilde{J}(z)\right)(t)-\tilde{\phi}(t)\right\|_{\mathbb{R}^n}
			& \leq M \left\|g(z_{\theta_1},\dots, z_{\theta_q})(0)\right\|_{\mathbb{R}^n}\\
			& + \int_{0}^{t}\left\|\boldsymbol{U}(t,s)f(s,z_s)\right\|_{\mathbb{R}^n} ds \\
			& \leq MN_{q}\sum_{i=1}^q\|z\|+ M \int_{0}^{t} \|f(s,z_s)\|ds \\
			& \leq MN_{q}q\|z\| + M t_1 \Psi(\|z||) \\
			& \leq MN_{q}q\left(\left\|\tilde{\phi}\right\|+\rho\right)+M \tau \Psi\left(\left\|\tilde{\phi}\right\|+\rho\right) \leq \rho.
		\end{split}
	\end{equation}
	
	Likewise, $t \in (s_i,t_{i+1}],\ i\in I_N$, gives
	\begin{equation} \label{eqT.c}
		\begin{split}
			\left\|\widetilde{F}\left(z,\widetilde{J}(z)\right)(t)-\tilde{\phi}(t)\right\|_{\mathbb{R}^n}
			& \leq \|\boldsymbol{U}(t,s_i)G_i(s_i,z(s_i))- \alpha \|_{\mathbb{R}^n} \\
			& + \int_{s_i}^{t}\left\|\boldsymbol{U}(t,s)f(s,z_s)\right\|_{\mathbb{R}^n} ds \\
			& \leq M L\|z\|+\| \alpha\|_{\mathbb{R}^n} + M (t_{i+1}-s_i) \Psi(\|z||) \\
			& \leq M L\left(\left\|\tilde{\phi}\right\|+\rho\right)+ \|\alpha\|_{\mathbb{R}^n} + M \tau \Psi\left(\left\|\tilde{\phi}\right\|+\rho\right)\\
			& \leq \rho.
		\end{split}
	\end{equation}
	
	Additionally, for $i\in I_N$, if $t \in (t_i,s_i]$, then,
	\begin{equation} \label{eqT.d}
		\begin{split}
			\left\|\widetilde{F}\left(z,\widetilde{J}(z)\right)(t)-\tilde{\phi}(t)\right\|_{\mathbb{R}^n} & = \|G_i(t,z(t)) -\beta \|_{\mathbb{R}^n} \\
			& \leq L\|z\|+\| \beta\|_{\mathbb{R}^n} \\
			&\leq L\left(\left\|\tilde{\phi}\right\|+\rho\right) +\| \beta\|_{\mathbb{R}^n}  \leq \rho.
		\end{split}
	\end{equation}
	
	Thus, equations \eqref{eqT.a} through \eqref{eqT.d} give
	\begin{equation*}
		\left\|\widetilde{F}\left(z,\widetilde{J}(z)\right)-\tilde{\phi}\right\|=\sup_{t\in\mathcal{J}}\left\|\widetilde{F}\left(z,\widetilde{J}(z)\right)(t)-\tilde{\phi}(t)\right\|_{\mathbb{R}^n}\leq\rho.
	\end{equation*}
	
	Applying Theorem \ref{Karakostas} to $\widetilde{J}$ and $\widetilde{F}$, it follows $
	\widetilde{F}\left(z,\widetilde{J}(z)\right) = z
	$, i.e., there exists a fixed-point solution $z\in D_{\rho}\subset PC_{r\tau}$, equivalent to the system \eqref{eqExist} solution given by Proposition \ref{characterization}.
\end{proof}

The following  theorem proves the uniqueness of the solution to system \eqref{eqExist}.

\begin{theorem}\label{uniqueness}
	Asumming  \textbf{H1} - \textbf{H4}, system  \eqref{eqExist} has a unique solution on $\mathcal{J}=[-r, \tau]$.
\end{theorem}

\begin{proof} Consider two solutions $z_1$ and $z_2$ to  \eqref{eqExist}, which satisfy \eqref{solution}. Let $\rho>0$ such that $z_1,\ z_2\in D_{\rho}$. Then, for $t \in [-r,0]$, the following estimate holds:
	\begin{equation} \label{eqT2.0}
		\begin{split}
			\|z_1(t)-z_2(t)\|_{\mathbb{R}^n}
			& \leq \left\|g\left(z_{2_{\theta_1}},\dots, z_{2_{\theta_q}}\right)(t)-g\left(z_{1_{\theta_1}},\dots, z_{1_{\theta_q}}\right)(t)\right\|_{\mathbb{R}^n} \\
			& \leq N_q q \|z_1-z_2\| \\
			& \leq \frac{1}{2}\|z_1-z_2\|.
		\end{split}
	\end{equation}
	
	If  $t \in (0,t_1]$, \textbf{H2} implies
	\begin{equation} \label{eqT2.1}
		\begin{split}
			\|z_1(t)-z_2(t)\|&_{\mathbb{R}^n} \\
			& \leq \left\|\boldsymbol{U}(t,0)\right\|\left\|g\left(z_{2_{\theta_1}},\ldots, z_{2_{\theta_q}}\right)(0)-g\left(z_{1_{\theta_1}},\ldots, z_{1_{\theta_q}}\right)(0)\right\|_{\mathbb{R}^n} \\
			& + \int_{0}^{t}\left\|\boldsymbol{U}(t,s)\left(
					f\left(s,z_{1_{s}}\right)-f(s,z_{2_s})\right)\right\|ds \\
			& \leq \left[M N_{q} q+ M t_1K\left(\|z_1\|,\|z_2\|\right)\right]\|z_1-z_2\| \\
			&\leq  \left[M N_{q} q+ M \tau K\left(\|\tilde{\phi}\|+\rho,\|\tilde{\phi}\| +\rho\right)\right] \|z_1-z_2\|,
		\end{split}    
	\end{equation}
	and, $t \in (s_{i},t_{i+1}]$, $i\in I_N$, yields
	\begin{equation} \label{eqT2.2}
		\begin{split}
			\|z_1(t)-z_2(t)\|_{\mathbb{R}^n}
			& \leq \|\boldsymbol{U}(t,s_i)\| \|G_i\left(s_i,z_1(s_i)\right) - G_i\left(s_i,z_2(s_i)\right) \|_{\mathbb{R}^n} \\
			& + \int_{s_i}^{t}\left\|\boldsymbol{U}(t,s)\left(f(s,z_{1_{s}})-f(s,z_{2_{s}})\right)\right\|ds \\
			& \leq \left[ M L + M (t_{i+1}-s_i)K\left(\|z_1\|, \|z_2\| \right) \right] \|z_1-z_2\| \\
			& \leq \left[ M L   + M \tau K\left(\left\|\tilde{\phi}\right\|+\rho, \left\|\tilde{\phi}\right\|+\rho \right) \right] \|z_1-z_2\|.
		\end{split}
	\end{equation}
	
	Lastly, if $t \in (t_i,s_i],\ i\in I_N$, then,
	\begin{equation} \label{eqT2.3}
		\begin{split}
			\|z_1(t)-z_2(t)\|_{\mathbb{R}^n}
			& \leq \left\|G_i(t,z_1(t)) - G_i(t,z_2(t))\right\|_{\mathbb{R}^n}  \\
			& \leq L \|z_1-z_2\| \\
			& \leq \frac{1}{2} \|z_1-z_2\|.
		\end{split}
	\end{equation}
	
	Therefore, taking the sup limit of equations \eqref{eqT2.0}-\eqref{eqT2.3} and \textbf{H4} imply that there exists a constant $m$, with $0<m<1$, such that:
	\begin{equation*}
		\begin{split}
			\|z_1-z_2\| & =\sup_{t\in \mathcal{J}}\|z_1(t)-z_2(t)\|_{\mathbb{R}^n}\leq m\|z_1-z_2\|.
		\end{split}
	\end{equation*}
	
	Hence, $z_1=z_2$.
\end{proof}

Finally, the next theorem and corollary extend the system solution towards $[-r, +\infty)$.

\begin{theorem}\label{prolongation}
	Assume \textbf{H1} - \textbf{H4} are satisfied, and consider the solution $z$ over a maximal interval $[-r,p_1)$. Then, $p_1= + \infty$, or there exists a convergent sequence $\left\{\tau_n\right\}_{n\in\mathbb{N}}$ to $p_1$, such that:
	\begin{equation}\label{seq}
		\lim_{n\rightarrow\infty} z(\tau_n) = \tilde{z}\in \partial B_{\left\|\tilde{\phi}\right\|+\rho}\subset \mathbb{R}^n .
	\end{equation}
\end{theorem}

\begin{proof}
	Assume $p_1<+\infty$, and suppose that there exists a neighborhood $V$ of the boundary of $B_{\left\|\tilde{\phi}\right\|+\rho}$ such that if $t\in [p_2,p_1)$, with $s_{N}<p_2<p_1$, then, $z(t)\notin V$.
	
	Without loss of generality, assume that $ V=B_{\left\|\tilde{\phi}\right\|+\rho} \backslash E$, with $E \subset  B_{\left\|\tilde{\phi}\right\|+\rho}$ a closed set and $z(t)\in E$, for $t \in [p_2,p_1)$.
	
	Consider $p_2\leq s<t<p_1$. It follows that:
	\begin{equation*}
		\begin{split}
			\|z(t)-z(s)\|_{\mathbb{R}^n}
			& \leq \left\|\boldsymbol{U}(t,s_N)-\boldsymbol{U}(s,s_N)\right\|\left\|G_N(s_N,z(s_N))\right\|_{\mathbb{R}^n}\\
			&  + \int_{s}^t\left\|\boldsymbol{U}(t,\xi)\right\|\left\|f(\xi,z_\xi)\right\|_{\mathbb{R}^n} d\xi\\
			&  + \int_{s_N}^s\left\|\boldsymbol{U}(t,\xi)-\boldsymbol{U}(s,\xi)\right\|\left\|f(\xi,z_\xi)\right\|_{\mathbb{R}^n} d\xi\\
			&\leq \left\|\boldsymbol{U}(t,s_N)-\boldsymbol{U}(s,s_N)\right\|L\|z\| + M(t-s)\Psi\left(\left\|\tilde{\phi}\right\|+\rho\right)\\
			& + \Psi\left(\left\|\tilde{\phi}\right\|+\rho\right) \int_{s_N}^s \left\|\boldsymbol{U}(t,\xi)-\boldsymbol{U}(s,\xi)\right\| d\xi .
		\end{split}
	\end{equation*}
	
	Then, uniform continuity of the evolution operator yields 
	\begin{equation*}
		\lim_{s\rightarrow p_1^-} \|z(t)-z(s)\|_{\mathbb{R}^n} = 0.
	\end{equation*}
	
	Thus, there exists $\tilde{z}\in \mathbb{R}^n$ such that $z(p_1^-)=\tilde{z}\in E$, and a solution can be defined at $p_1$ through extending $z$ by continuity, which contradicts the  maximality of $[-r,p_1)$. Thus, either $p_1=+\infty$, or a sequence $\{\tau_n\}_n$ exists and fulfills \eqref{seq}.
\end{proof}

\begin{corollary} \label{Corollary_prol}
	Under Theorem \ref{prolongation} assumptions, suppose that:
	\begin{equation*}
		\|f(t,\varphi)\|_{\mathbb{R}^n}   \leq h(t)\left(1+\|\varphi (0)\|_{\mathbb{R}^n}\right),
	\end{equation*}
	for $\varphi \in  PC_{r}$ and $h: \overline{\mathbb{R}_+}\longrightarrow\overline{\mathbb{R}_+}$ continuous. Then, there exists a unique solution to problem \eqref{eqExist} on $[-r,+\infty)$.
\end{corollary}

\begin{proof}
	Consider $t\in [s_N, p_1)$. It follows that:
	\begin{equation*}
		\begin{split}
			\|z(t)\|_{\mathbb{R}^n} &\leq \left\|\boldsymbol{U} (t,s_N)\right\|\left\|G_N(s_N, z(s_N))\right\|_{\mathbb{R}^n} +\int_{s_N}^t\left\|\boldsymbol{U}(t,s)\right\|\left\|f(s,z_s)\right\|ds\\
			&\leq ML\|z(s_N)\|_{\mathbb{R}^n} +   \int_{s_N}^{p_1}M h(s)ds+ \int_{s_N}^{t} M h(s)\|z (s)\|_{{\mathbb{R}^n}} ds .
		\end{split}
	\end{equation*}
	
	Hence, Gr{\"o}nwall's inequality yields
	\begin{equation*}
		\|z(t)\|_{\mathbb{R}^n}\leq M\left(L\|z(s_N)\|_{\mathbb{R}^n} +   \int_{s_N}^{p_1} h(s)ds\right)\exp\left( \int_{s_N}^{p_1} M h(s)ds\right).
	\end{equation*}
	
	By Theorem \ref{prolongation}, the solution stays bounded, as desired.
\end{proof}

\section{Example} \label{Example}

In this section, particular definitions for functions $G_i$, $g$ and $f$, $i\in I_N$, exemplify the results of this work. To this end, consider an arbitrary finite-dimensional continuous operator $\boldsymbol{A}$, such that $\boldsymbol{A}(t)$ is a $n\times n$ matrix. 

Given $N, R \in \mathbb{N}$, the non-linear term, $f:\overline{\mathbb{R}_+} \times PC_r([-r,0];\mathbb{R}^n) \longrightarrow  \mathbb{R}^n$, the functions describing non-instantaneous impulses, $G_i: (t_{i},s_{i}] \times \mathbb{R}^n  \longrightarrow \mathbb{R}^n$, and non-local conditions, $	g:PC_r^q([-r,0];(\mathbb{R}^n)^q)\longrightarrow PC_r([-r,0];\mathbb{R}^n)$, are given as follows, for $z\in PC_{r\tau}$ and $i\in I_N$,
\begin{equation*}
	\begin{split}
		f(t, \varphi)=\frac{1}{R}\left(\begin{array}{ccc}
				(\varphi_{1}(-r))^2 \\
				(\varphi_2(-r))^2 \\
				\vdots\\
				(\varphi_n(-r))^2
		\end{array}\right),\quad
	G_i(t,z(t))=\frac{\cos{(s_i)}}{R} \left(\begin{array}{ccc}
				\sin{(z_{1}(t))} \\
				\sin{(z_{2}(t))} \\
				\vdots\\
				\sin{(z_{n}(t))}
		\end{array}\right),
	\end{split}
\end{equation*}
\begin{equation*}
	\begin{split}
	g(\varphi)=\sum_{i=1}^{q} \frac{1}{R} \varphi_i.
	\end{split}
\end{equation*}

Clearly, $g$ verifies that
$g(0)=0$, and if $t\in [-r,0]$, then,
\begin{equation*}
	\begin{split}
		\|g(y)(t)-g(z)(t)\|_{\mathbb{R}^n} \leq \frac{1}{R} \|y(t)-z(t)\|_{(\mathbb{R}^n)^q}, \quad \text{ for all } y,\ z \in PCp^q.
	\end{split}
\end{equation*}

The functions $G_i$, $i\in I_N$, satisfy $G_i(\cdot,0) = 0,$ and, for any $y,\ z\in PC_{r\tau}$, given $t\in (t_i,s_i]$,
\begin{equation*}
	\begin{split}
		\left\|G_i(t,y(t))-G_i(t,z(t))\right\|_{\mathbb{R}^n} & \leq \frac{|\cos(s_i)|}{R} \left(\sum_{k=1}^N |\sin(y_k(t))-\sin(z_k(t))|^2\right)^{1/2}\\
		& \leq \frac{|\cos(s_i)|}{R}\|y-z\|.
	\end{split}
\end{equation*}

For $R$ sufficiently large, it yields
\begin{equation*}
\displaystyle \frac{|\cos(s_i)|}{R} + \frac{1}{R} q  < \frac{1}{2}.
\end{equation*}

Finally, given $t\geq 0$, $y,\ z\in PC_{r\tau}$, and $\varphi \in PC_{r}$, the function $f$ satisfies
\begin{equation*}
	\begin{split}
		\|f(t,y_t)-f(t,z_t)\|_{\mathbb{R}^n}& \leq \frac{1}{R}\left(\sum_{k=1}^{n}  \
			\left(|y_k(t-r)|+|z_k(t-r)|\right)^{2} \right)^{1/2} \|y-z\|\\
		&\leq K(\|y\|,\|z\|)\|y-z\|,
	\end{split}
\end{equation*}
and 
\begin{equation*}
	\|f(t,\varphi)\|_{\mathbb{R}^n}=\frac{1}{R}\left\|\left(\begin{array}{ccc}
			(\varphi_{1}(-r))^2 \\
			(\varphi_2(-r))^2 \\
			\vdots\\
			(\varphi_n(-r))^2
		\end{array}\right)\right\|_{\mathbb{R}^n}\leq\Psi(\|\varphi\|),
\end{equation*}
where $K$ and $\Psi$ are continuous non-decreasing functions. Hence, hypotheses \textbf{H1} and \textbf{H2} are satisfied. For $R$ sufficiently large, conditions \textbf{H3} and \textbf{H4} are similarly verified. Then, by Theorem \ref{uniqueness} and Corollary \ref{Corollary_prol}, system \eqref{eqExist}, with the foregoing definitions, admits a unique solution on $[-r,+\infty)$.

\section{Final Remarks} \label{Final Remarks}

In this work, existence and uniqueness of solutions for semi-linear systems of non-autonomous differential equations considering non-instantaneous impulses, delay, and non-local conditions simultaneously were proved. The technique used was based on Karakostas' fixed-point theorem, by transforming the existence of solutions problem into a fixed-point existence problem of a certain operator equation satisfying the specific conditions. This led to choose the adequate hypotheses to meet the requirements of that theorem. Observe that this work can be generalized to infinite-dimensional Banach spaces. However, proving equicontinuity of specific operator families and the main operator compactness must be carefully treated before applying a fixed-point theorem. The strongly continuous semigroup in the non-autonomous system requires compactness to ensure the uniform continuity away from zero. A different version of Arzel\`a-Ascoli theorem must be considered on the corresponding functional spaces. 
To end, the controllability of these systems is part of our outgoing research. In particular, the exact and approximate controllability of this system can be proven using Rothe's fixed-point theorem \cite{Leiva2014} and the techniques developed in \cite{Bashirov2015}.

\bibliographystyle{\mmnbibstyle}
\bibliography{mmnsample.bib}

\end{document}